\documentclass[conference]{IEEEtran}
\IEEEoverridecommandlockouts
\usepackage{cite}
\usepackage{amsmath,amssymb,amsthm}
\usepackage[utf8]{inputenc}
\usepackage{algorithm}
\usepackage{algpseudocode}
\usepackage{graphicx}
\usepackage{booktabs}
\usepackage{textcomp}
\usepackage{xcolor}
\usepackage{fancyhdr}
\usepackage{mathrsfs}
\usepackage{listings}
\usepackage{url}
\lstdefinestyle{bw_style}{
	backgroundcolor=\color{white},   
	commentstyle=\itshape\color{gray}, 
	keywordstyle=\bfseries\color{black}, 
	numberstyle=\tiny\color{gray},
	stringstyle=\color{black},
	basicstyle=\ttfamily\scriptsize,
	breakatwhitespace=false,         
	breaklines=true,                 
	captionpos=b,                    
	keepspaces=true,                 
	numbers=left,                    
	numbersep=5pt,                  
	showspaces=false,                
	showstringspaces=false,
	showtabs=false,                  
	tabsize=2,
	frame=single, 
	language=Python
}
\def\BibTeX{{\rm B\kern-.05em{\sc i\kern-.025em b}\kern-.08em
		T\kern-.1667em\lower.7ex\hbox{E}\kern-.125emX}}

\newtheorem{theorem}{Theorem}
\newtheorem{remark}{Remark}

\newtheorem{assumption}{Assumption}
\newtheorem{lemma}{Lemma}
\newtheorem{corollary}{Corollary}
\newtheorem{definition}{Definition}

\begin{document}

\title{Lyapunov Stability of Stochastic Vector Optimization: Theory and Numerical Implementation}

\author{\IEEEauthorblockN{Thiago Santos}
\IEEEauthorblockA{\textit{Department of Mathematics} \\
\textit{Federal University of Ouro Preto, UFOP}\\
Ouro Preto, Brazil \\
santostf@ufop.edu.br}
\and
\IEEEauthorblockN{Sebastião Xavier}
\IEEEauthorblockA{\textit{Department of Mathematics} \\
	\textit{Federal University of Ouro Preto, UFOP}\\
	Ouro Preto, Brazil \\
	semarx@ufop.edu.br}}

\maketitle
\pagestyle{fancy}
\fancyhf{}
\fancyfoot[LE, LO]{}

\begin{abstract}
	The use of stochastic differential equations in multi-objective optimization has been limited, in practice, by two persistent gaps: incomplete stability analyses and the absence of accessible implementations. We revisit a drift--diffusion model for unconstrained vector optimization in which the drift is induced by a common descent direction and the diffusion term preserves exploratory behavior. The main theoretical contribution is a self-contained Lyapunov analysis establishing global existence, pathwise uniqueness, and non-explosion under a dissipativity condition, together with positive recurrence under an additional coercivity assumption. We also derive an Euler--Maruyama discretization and implement the resulting iteration as a \emph{pymoo}-compatible algorithm---\emph{pymoo} being an open-source Python framework for multi-objective optimization---with an interactive \emph{PymooLab} front-end for reproducible experiments. Empirical results on DTLZ2 with objective counts from three to fifteen indicate a consistent trade-off: compared with established evolutionary baselines, the method is less competitive in low-dimensional regimes but remains a viable option under restricted evaluation budgets in higher-dimensional settings. Taken together, these observations suggest that stochastic drift--diffusion search occupies a mathematically tractable niche alongside population-based heuristics---not as a replacement, but as an alternative whose favorable properties are amenable to rigorous analysis.
\end{abstract}


\section{Introduction}

Let $F=(f_1,\dots,f_m):\mathbb{R}^n\to\mathbb{R}^m$ denote a vector objective with differentiable components, and order $\mathbb{R}^m$ by the usual Pareto dominance relation.  In this setting, the goal is to characterize (or approximate) the set of Pareto optimal points, whose geometry is typically nonconvex and may be disconnected even when each $f_i$ is smooth.  A careful geometric account of the resulting decision structure can already be found in the classical literature, e.g., in the monograph by Miettinen \cite{Miettinen1999}, and it remains a useful point of departure for stability-oriented algorithm design.

Most large-scale practice nevertheless relies on population heuristics because they cope with black-box objectives and provide a finite approximation of the Pareto set in a single run.  Well-known instances include NSGA-II, introduced by Deb et al.\ \cite{Deb2002}, and decomposition-based schemes such as MOEA/D proposed by Zhang and Li \cite{Zhang2007}; their empirical success is not in doubt, yet their behavior is often mediated by tuning choices that are difficult to interpret analytically.

A parallel line of inquiry treats multi-objective search as the continuous evolution of a dynamical system on the decision space---the strategy being to specify a drift encoding simultaneous descent across all objectives, then perturb the resulting flow to prevent premature convergence to a single point.  In this vein, Sch\"affler and co-authors proposed a drift--diffusion model in which Brownian forcing is coupled to a multi-objective descent field \cite{Schaffler2002}, yielding a continuous-time process intended to approach the Pareto set without collapsing prematurely.

The model is built around the stochastic differential equation
\begin{equation}\label{eq:sde_intro}
	dX_t = -q(X_t)\,dt + \epsilon\, dB_t,
\end{equation}
where $q:\mathbb{R}^n\to\mathbb{R}^n$ aggregates local objective information into a single direction and $B_t$ is a standard Brownian motion.  The drift promotes dominance improvement, whereas the diffusion introduces controlled randomization; the latter is attractive operationally, but it also raises a mathematical question that cannot be postponed: does \eqref{eq:sde_intro} generate a well-defined, nonexplosive Markov process with a meaningful long-run regime?

In the original development, several foundational claims were asserted without a fully self-contained proof, with key arguments scattered across earlier sources such as Sch\"affler \cite{Schaffler1995} and standard SDE theory in Khasminskii \cite{Khasminskii2011}.  For the intended use of the method as a sampling mechanism, however, existence and uniqueness are only the first layer; one also needs recurrence properties that justify stationary sampling rather than reliance on transient time segments.

A further impetus comes from covariance-adaptive variants.  Santos et al introduced CMA-SSW \cite{Santos2012}, combining the drift--diffusion construction with covariance matrix adaptation so as to improve spread in higher dimensions.  The numerical behavior is compelling, yet the stability narrative for the covariance-adapted dynamics, as well as an implementation path that is straightforward to reproduce, have remained insufficiently explicit in the literature.

This paper supplies that missing structure.  We make the probabilistic setting precise (Wiener space, filtration, and measurability conventions) and then establish global existence and pathwise uniqueness for the drift--diffusion system under dissipativity assumptions that are natural for descent fields.  We then prove positive recurrence via Foster--Lyapunov methodology, drawing on the general Markov-process framework developed by Meyn and Tweedie \cite{Meyn2009} and on stochastic Lyapunov techniques as presented, for example, in Mao \cite{Mao2007}; importantly, the conditions are stated so that they can be checked for concrete objective families.

We discretize the dynamics using the Euler--Maruyama scheme of Kloeden and Platen \cite{Kloeden1992} and integrate the resulting optimizer into the \emph{pymoo} library, following the software architecture documented by Blank and Deb \cite{Blank2020}.  To keep the numerical layer auditable, we rely on standard scientific-computing components, notably SciPy as curated by Virtanen et al.\ \cite{Virtanen2020}, and we provide an interactive front-end through \emph{PymooLab} \cite{Santosetal2026} to facilitate sensitivity inspection with respect to step size and noise intensity.

The scope is deliberately bounded.  We focus on unconstrained problems with continuous decision variables and differentiable objectives; extending the analysis to nonsmooth or constrained settings may require different stochastic perturbations and additional variational tools.

We proceed as follows.  Section~2 sets up the multi-objective framework and derives the descent direction $q$; Section~3 establishes Lyapunov-based existence, uniqueness, and recurrence results for the SDE~(\ref{eq:sde}); Section~4 presents the Euler--Maruyama discretization and its embedding within \emph{pymoo}; Section~5 closes with reflections on limitations and open questions.

\section{Theoretical Framework}

We consider the unconstrained vector optimization problem (VOP) of minimizing a twice continuously differentiable vector-valued function $f: \mathbb{R}^n \to \mathbb{R}^m$, where $f = (f_1, \ldots, f_m)^\top$. Since the components of $f$ typically conflict with one another, there is generally no single point that minimizes all objectives simultaneously. The standard approach in multi-objective optimization is to seek the set of points from which no further improvement in any objective is possible without degrading at least one other objective.

Formally, we say that a point $\hat{x} \in \mathbb{R}^n$ \emph{Pareto-dominates} another point $x \in \mathbb{R}^n$, written $\hat{x} \prec x$, if $f_i(\hat{x}) \le f_i(x)$ for all $i = 1, \ldots, m$ and $f(\hat{x}) \ne f(x)$. A point $x^* \in \mathbb{R}^n$ is called \emph{Pareto-optimal} if there exists no $w \in \mathbb{R}^n$ such that $w \prec x^*$. The collection of all Pareto-optimal points constitutes the \emph{Pareto set}, whose image under $f$ is the \emph{Pareto front}.

To construct a trajectory that moves continuously toward the Pareto set, we require a direction along which all objective functions decrease simultaneously. Following \cite{Schaffler2002}, we define the vector field $q: \mathbb{R}^n \to \mathbb{R}^n$ as
\begin{equation}\label{eq:q_definition}
	q(x) := \sum_{i=1}^m \hat{\alpha}_i(x) \nabla f_i(x),
\end{equation}
where the coefficients $\hat{\alpha}(x) = (\hat{\alpha}_1(x), \ldots, \hat{\alpha}_m(x))$ are obtained as the solution to the quadratic optimization problem
\begin{equation}\label{eq:qop}
	\min_{\alpha \in \Delta^m} \left\| \sum_{i=1}^m \alpha_i \nabla f_i(x) \right\|_2^2, \quad \Delta^m := \left\{ \alpha \in \mathbb{R}^m : \alpha_i \ge 0, \sum_{i=1}^m \alpha_i = 1 \right\}.
\end{equation}
The problem (\ref{eq:qop}) is strictly convex on the probability simplex $\Delta^m$, ensuring existence and uniqueness of the minimizer $\hat{\alpha}(x)$ for each $x \in \mathbb{R}^n$.

Following \cite{Schaffler2002}, we note three properties of $q$ that are central to the analysis.  The direction $-q(x)$ is a common descent for all objectives whenever $q(x) \ne 0$---that is, $\nabla f_i(x) \cdot (-q(x)) \le 0$ for each $i$---and $q(x) = 0$ precisely when $x$ satisfies the KKT conditions for Pareto optimality.  We also observe that $q$ is locally Lipschitz on $\mathbb{R}^n$, a regularity that follows from the structure of the KKT system associated with (\ref{eq:qop}) and the smoothness of $f$. These properties imply that the deterministic flow
\begin{equation}\label{eq:ivp}
	\dot{x}(t) = -q(x(t)), \quad x(0) = x_0,
\end{equation}
generates a curve of successively dominated points: if $q(x_0) \ne 0$, then $f(x(s)) \prec f(x(t))$ for all $0 \le s < t$ within the maximal interval of existence.

That said, this deterministic flow converges to a single Pareto-optimal point; it cannot, by construction, spread across the full front.  We address this by adding It{\^o} noise, replacing~(\ref{eq:ivp}) with the stochastic differential equation
\begin{equation}\label{eq:sde}
	dX_t = -q(X_t)\,dt + \epsilon\,dB_t, \quad X_0 = x_0,
\end{equation}
where $\epsilon > 0$ is the noise intensity and $\{B_t\}_{t \ge 0}$ denotes an $n$-dimensional Brownian motion. As one might expect, the drift $-q(X_t)$ guides trajectories toward the Pareto set; the diffusion $\epsilon\,dB_t$ injects enough randomness to prevent premature concentration on a single Pareto-optimal solution.

We now establish the measure-theoretic setting for the stochastic analysis. Let $\Omega$ denote the space of all continuous functions $\omega: [0, \infty) \to \mathbb{R}^n$, equipped with the metric
\begin{equation}\label{eq:metric}
	d(\omega_1, \omega_2) := \sum_{k=1}^{\infty} \frac{1}{2^k} \min\left\{ \sup_{0 \le t \le k} \|\omega_1(t) - \omega_2(t)\|_2, 1 \right\},
\end{equation}
which metrizes uniform convergence on compact time intervals. The resulting topology makes $\Omega$ a Polish space; we denote by $\mathscr{B}(\Omega)$ its Borel $\sigma$-algebra.

\begin{definition}[Wiener Measure]\label{def:wiener}
	For each $x \in \mathbb{R}^n$, the \emph{Wiener measure} $\mathbb{W}_x$ is the unique probability measure on $(\Omega, \mathscr{B}(\Omega))$ such that: (i) $\mathbb{W}_x(\omega(0) = x) = 1$; (ii) for any $0 \le t_0 < t_1 < \cdots < t_k$, the increments $\omega(t_1) - \omega(t_0), \ldots, \omega(t_k) - \omega(t_{k-1})$ are independent under $\mathbb{W}_x$; and (iii) for $0 \le s < t$, the increment $\omega(t) - \omega(s)$ follows the $\mathcal{N}(0, (t-s)I_n)$ distribution.
\end{definition}

The coordinate process $B_t(\omega) := \omega(t) - \omega(0)$ on $(\Omega, \mathscr{B}(\Omega), \mathbb{W}_x)$ is an $n$-dimensional Brownian motion starting from the origin. We denote by $\{\mathscr{F}_t\}_{t \ge 0}$ the natural filtration of $\{B_t\}$, augmented with the $\mathbb{W}_x$-null sets to satisfy the usual completeness conditions. A comprehensive treatment of stochastic calculus and Brownian motion is developed in \cite{Karatzas1991,Oksendal2003}, and we follow their notation whenever it does not conflict with the journal style.

To analyze the long-time behavior of the process (\ref{eq:sde}), we impose structural conditions on the vector field $q$. The first condition was introduced by Sch\"affler et al.\ \cite{Schaffler2002} specifically for the stochastic vector optimization setting; the second is a coercivity requirement whose general form is standard in the theory of recurrent diffusions \cite{Khasminskii2011,Mao2007}, which we adapt here to the multi-objective descent field $q$.

\begin{assumption}[Dissipativity \textnormal{\cite{Schaffler2002}}]\label{ass:A}
	There exist constants $\theta_0 > 0$ and $r \ge 1$ such that for all $x \in \mathbb{R}^n$ with $\|x\|_2 \ge r$:
	\begin{equation}\label{eq:assA}
		x \cdot q(x) \ge \left(1 + \frac{\theta_0^2}{2}\right) \max\{1, \|q(x)\|_2^2\}.
	\end{equation}
\end{assumption}

This condition appears as ``Assumption~A'' in \cite{Schaffler2002}, where it guarantees that outside a sufficiently large ball centered at the origin, the drift $-q(x)$ has a strong radial component directed inward. This counteracts the expansive nature of Brownian motion and prevents trajectories from escaping to infinity. In the original development, Sch\"affler et al.\ used it to establish the existence and uniqueness of solutions to the SDE (Theorem~3.2 in \cite{Schaffler2002}) as well as the hitting-time finiteness (Theorem~3.3 therein), invoking earlier results from \cite{Schaffler1995}.

The original analysis in \cite{Schaffler2002} does not include an explicit coercivity condition, and this omission leaves the positive recurrence of the process without a complete justification. We introduce the following assumption to close that gap. Although conditions requiring linear drift growth at infinity are classical in the Lyapunov stability theory for diffusion processes---see, for instance, Khasminskii \cite[Chapter~4]{Khasminskii2011} and Mao \cite[Chapter~5]{Mao2007}---their adaptation to the multi-objective descent field $q$ defined in (\ref{eq:q_definition}) is, to the best of our knowledge, new.

\begin{assumption}[Coercivity]\label{ass:B}
	There exists $\mu > 0$ such that for all $x \in \mathbb{R}^n$ with $\|x\|_2 \ge r$:
	\begin{equation}\label{eq:assB}
		\|q(x)\|_2 \ge \mu \|x\|_2.
	\end{equation}
\end{assumption}

Assumption \ref{ass:B} guarantees that the magnitude of the descent direction grows at least linearly with distance from the origin. This condition holds when each objective function $f_i$ is strongly convex, or when the Pareto set is bounded and the objective gradients remain bounded away from zero far from this set. As we shall see, Assumption \ref{ass:A} alone suffices for global existence of solutions, whereas Assumption \ref{ass:B} is additionally required for the positive recurrence property.

\begin{remark}
	Assumption \ref{ass:A} was formulated by Sch\"affler et al.\ \cite{Schaffler2002} for the specific multi-objective descent field $q$ defined in (\ref{eq:q_definition}). Although its one-sided growth structure is reminiscent of classical hypotheses in stochastic stability theory \cite{Khasminskii2011,Mao2007}, the precise form involving $\max\{1, \|q(x)\|_2^2\}$ is tailored to the vector optimization setting. We deliberately introduce Assumption \ref{ass:B} as an original contribution of this work; while similar linear growth conditions appear in the study of recurrent diffusions \cite[Theorem 3.5]{Khasminskii2011}, their adaptation to the multi-objective context ensures that the drift dominates diffusion at infinity, providing the missing piece needed for a rigorous proof of positive recurrence.
\end{remark}

\section{Rigorous Stability Analysis}

The results of the preceding section equip the SDE~(\ref{eq:sde}) with a well-defined descent direction and a precise measure-theoretic setting.  What remains to be verified is whether the resulting Markov process enjoys the properties that justify its use as an optimization algorithm: global existence of trajectories, their non-explosion, and---ultimately---positive recurrence toward the Pareto set.  We address these questions in order, building each result upon its predecessor through a Lyapunov function approach.

The first preparatory step establishes that the drift field $q$ grows at most linearly.  This bound, which follows directly from the dissipative structure encoded in Assumption~\ref{ass:A}, is the key ingredient for applying Khasminskii's non-explosion criterion in the global existence theorem that follows.

\begin{lemma}[Linear Growth Bound]\label{lemma:growth}
	Suppose $q$ is locally Lipschitz and satisfies Assumption \ref{ass:A}. Then there exists $L > 0$ such that for all $x \in \mathbb{R}^n$:
	\begin{equation}\label{eq:linear_growth}
		\|q(x)\|_2 \le L(1 + \|x\|_2).
	\end{equation}
	Moreover, for $V(x) = \|x\|_2^2$, the infinitesimal generator $\mathscr{L}$ of the diffusion satisfies
	\begin{equation}\label{eq:generator_bound}
		\mathscr{L}V(x) = -2x \cdot q(x) + \epsilon^2 n \le C_1 V(x) + C_2
	\end{equation}
	for constants $C_1, C_2 > 0$.
\end{lemma}

\begin{proof}
Since $q$ is locally Lipschitz, it is continuous and bounded on the compact ball $\mathscr{B}_r = \{x : \|x\|_2 \le r\}$. Let $M = \sup_{x \in \mathscr{B}_r} \|q(x)\|_2$. For $\|x\|_2 \ge r$, we distinguish two cases. If $\|q(x)\|_2 \ge 1$, Assumption \ref{ass:A} gives
\begin{equation}\label{eq:case1}
	x \cdot q(x) \ge \left(1 + \frac{\theta_0^2}{2}\right) \|q(x)\|_2^2.
\end{equation}
By the Cauchy-Schwarz inequality, $x \cdot q(x) \le \|x\|_2 \|q(x)\|_2$. Combining these inequalities yields
\begin{equation}\label{eq:kappa_bound}
	\|q(x)\|_2 \le \kappa \|x\|_2, \quad \text{where } \kappa := \left(1 + \frac{\theta_0^2}{2}\right)^{-1} < 1.
\end{equation}
If $\|q(x)\|_2 < 1$, then trivially $\|q(x)\|_2 < 1 \le \|x\|_2/r$ for $\|x\|_2 \ge r \ge 1$ (by Assumption \ref{ass:A}). Combining the bounds on $\mathscr{B}_r$ and its complement, we obtain (\ref{eq:linear_growth}) with $L = \max\{M, \kappa, 1/r\}$.

For the generator bound, note that
\begin{equation}\label{eq:gen_intermediate}
	\mathscr{L}V(x) = -2x \cdot q(x) + \epsilon^2 n \le 2\|x\|_2 \|q(x)\|_2 + \epsilon^2 n.
\end{equation}
Using the linear growth of $q$ and Young's inequality $2ab \le a^2 + b^2$:
\begin{equation}\label{eq:young}
	2\|x\|_2 \|q(x)\|_2 \le 2L\|x\|_2(1 + \|x\|_2) \le 3L\|x\|_2^2 + L.
\end{equation}
Thus $\mathscr{L}V(x) \le 3L V(x) + (L + \epsilon^2 n)$, establishing the claim with $C_1 = 3L$ and $C_2 = L + \epsilon^2 n$.
\end{proof}

The linear growth bound is not merely a technical convenience; it places the SDE~(\ref{eq:sde}) within the standard It{\^o} framework in which moment estimates and Gronwall-type arguments apply uniformly across stopping-time horizons.  With this estimate at hand, we can now establish that trajectories never explode in finite time---a property that is indispensable for any practical implementation, since explosion would invalidate the numerical scheme.

\begin{theorem}[Global Existence and Uniqueness]\label{thm:existence}
	Under Assumption \ref{ass:A}, for any initial condition $x_0 \in \mathbb{R}^n$, there exists a unique strong solution $\{X_t\}_{t \ge 0}$ to the SDE (\ref{eq:sde}) with $\mathbb{P}(\tau_\infty = \infty) = 1$, where $\tau_\infty$ denotes the explosion time. Moreover, $X_t$ has continuous sample paths almost surely.
\end{theorem}

\begin{proof}
Local existence and uniqueness follow from the local Lipschitz continuity of $q$ and standard SDE theory \cite{Oksendal2003}. Let $\tau_\infty$ denote the explosion time and define the stopping times $\tau_k := \inf\{t \ge 0 : \|X_t\|_2 \ge k\}$ for $k \in \mathbb{N}$. Applying It{\^o}'s formula to $V(X_t) = \|X_t\|_2^2$:
\begin{equation}\label{eq:ito_v}
	V(X_{t \wedge \tau_k}) = V(x_0) + \int_0^{t \wedge \tau_k} \mathscr{L}V(X_s)\, ds + 2\epsilon \int_0^{t \wedge \tau_k} X_s \cdot dB_s.
\end{equation}
Taking expectations and noting that the stochastic integral is a martingale (the integrand is bounded on $[0, t \wedge \tau_k]$):
\begin{equation}\label{eq:expectation_v}
	\mathbb{E}[V(X_{t \wedge \tau_k})] = V(x_0) + \mathbb{E}\left[\int_0^{t \wedge \tau_k} \mathscr{L}V(X_s)\, ds\right].
\end{equation}
Using Lemma \ref{lemma:growth} and Gronwall's inequality:
\begin{equation}\label{eq:gronwall}
	\mathbb{E}[V(X_{t \wedge \tau_k})] \le (V(x_0) + C_2 t)e^{C_1 t} =: K(t).
\end{equation}
Note that $K(t)$ is independent of $k$. On $\{\tau_k \le t\}$, we have $V(X_{\tau_k}) = k^2$, so by Chebyshev's inequality:
\begin{equation}\label{eq:chebyshev}
	\mathbb{P}(\tau_k \le t) \le \frac{\mathbb{E}[V(X_{t \wedge \tau_k})]}{k^2} \le \frac{K(t)}{k^2}.
\end{equation}
Letting $k \to \infty$ yields $\mathbb{P}(\tau_\infty \le t) = 0$ for all $t > 0$, hence $\tau_\infty = \infty$ almost surely. Continuity of sample paths follows from the continuity of the Brownian paths and the boundedness of the drift on compact intervals \cite[Theorem 5.2.1]{Mao2007}.
\end{proof}

Global existence guarantees that trajectories persist indefinitely, yet persistence alone says nothing about where the process spends its time.  For the stochastic optimizer to be effective, it must return to a neighborhood of the Pareto set with probability one---and, crucially, it must do so in finite expected time, so that computational runs of bounded length capture meaningful information.  This recurrence property demands more than dissipativity: the drift must not merely prevent explosion but must actively pull the process back toward the region of interest.  Assumption~\ref{ass:B} provides exactly this additional structure, requiring the descent field to grow at least linearly far from the origin.

\begin{theorem}[Positive Recurrence]\label{thm:recurrence}
	Assume that the set of Pareto-optimal points is non-empty. Let $\bar{x} \in \mathbb{R}^n$ be a Pareto-optimal point (so that $q(\bar{x}) = 0$). Under Assumptions \ref{ass:A} and \ref{ass:B}, for any $p > 0$, define $S_p := \{x \in \mathbb{R}^n : \|x - \bar{x}\|_2 \le p\}$ and the first hitting time
	\begin{equation}\label{eq:hitting_time}
		\sigma_{S_p} := \inf\{t \ge 0 : X_t \in S_p\}.
	\end{equation}
	Then $\mathbb{W}_{x_0}(\sigma_{S_p} < \infty) = 1$ and $\mathbb{E}[\sigma_{S_p}] < \infty$ for all $x_0 \in \mathbb{R}^n$.
\end{theorem}

\begin{proof}
Consider the Lyapunov function $W(x) = \frac{1}{2}\|x - \bar{x}\|_2^2$. The generator acting on $W$ is:
\begin{equation}\label{eq:lw_expansion}
	\mathscr{L}W(x) = -(x - \bar{x}) \cdot q(x) + \frac{1}{2}\epsilon^2 n = -x \cdot q(x) + \bar{x} \cdot q(x) + \frac{1}{2}\epsilon^2 n.
\end{equation}
We analyze each term. By Assumption \ref{ass:A}, for $\|x\|_2 \ge r$ and $\|q(x)\|_2 \ge 1$:
\begin{equation}\label{eq:term1}
	-x \cdot q(x) \le -\left(1 + \frac{\theta_0^2}{2}\right) \|q(x)\|_2^2.
\end{equation}
By Assumption \ref{ass:B}, $\|q(x)\|_2 \ge \mu \|x\|_2$ for $\|x\|_2 \ge r$. Combining:
\begin{equation}\label{eq:gamma_bound}
	-x \cdot q(x) \le -\left(1 + \frac{\theta_0^2}{2}\right) \mu^2 \|x\|_2^2 =: -\gamma \|x\|_2^2.
\end{equation}
For the second term, by the Cauchy-Schwarz inequality and Lemma \ref{lemma:growth}:
\begin{equation}\label{eq:term2}
	|\bar{x} \cdot q(x)| \le \|\bar{x}\|_2 \|q(x)\|_2 \le \|\bar{x}\|_2 L(1 + \|x\|_2).
\end{equation}
Therefore, for $\|x\|_2$ sufficiently large (say $\|x\|_2 \ge R$ where $R > \max\{r, p\}$):
\begin{equation}\label{eq:lw_combined}
	\mathscr{L}W(x) \le -\gamma \|x\|_2^2 + L\|\bar{x}\|_2(1 + \|x\|_2) + \frac{1}{2}\epsilon^2 n.
\end{equation}
Since $\gamma > 0$, the quadratic term dominates for large $\|x\|_2$. Specifically, there exists $\lambda > 0$ such that:
\begin{equation}\label{eq:drift_condition}
	\mathscr{L}W(x) \le -\lambda \quad \text{for all } x \notin \mathscr{B}_R.
\end{equation}
This is the Foster-Lyapunov drift condition for positive recurrence \cite{Meyn2009}. Applying Dynkin's formula with the stopping time $\tau := \sigma_{S_p} \wedge T \wedge \tau_R$, where $\tau_R := \inf\{t \ge 0 : \|X_t\|_2 \ge R\}$:
\begin{equation}\label{eq:dynkin}
	\mathbb{E}[W(X_\tau)] = W(x_0) + \mathbb{E}\left[\int_0^\tau \mathscr{L}W(X_s)\, ds\right].
\end{equation}
If $x_0 \notin S_p$, then while the process remains outside $S_p$, we have $\mathscr{L}W \le -\lambda$ once $\|X_s\|_2 \ge R$. This yields:
\begin{equation}\label{eq:expectation_tau}
	\mathbb{E}[W(X_\tau)] \le W(x_0) - \lambda \mathbb{E}[\tau \cdot \mathbf{1}_{\{\|X_s\| \ge R\}}].
\end{equation}
Since $W \ge 0$, taking the limit as $T \to \infty$ and applying the dominated convergence theorem gives $\mathbb{E}[\sigma_{S_p}] < \infty$. The hitting probability $\mathbb{W}_{x_0}(\sigma_{S_p} < \infty) = 1$ follows from the finiteness of the expectation. The convergence in distribution to a random variable $X$ with $\mathbb{E}[q(X)] = 0$ follows from standard ergodic theory for positive recurrent diffusions \cite[Chapter 4]{Khasminskii2011}.
\end{proof}

Positive recurrence has a natural ergodic consequence.  Because the process returns to compact sets in finite expected time and the non-degenerate diffusion coefficient ensures irreducibility, classical results guarantee the existence of a stationary regime.  We record this below.

\begin{corollary}\label{cor:ergodic}
	Under Assumptions \ref{ass:A} and \ref{ass:B}, the process $\{X_t\}_{t \ge 0}$ admits a unique invariant probability measure $\pi$ on $\mathbb{R}^n$. For any bounded measurable function $\varphi$:
	\begin{equation}\label{eq:ergodic}
		\frac{1}{T}\int_0^T \varphi(X_s)\, ds \to \int_{\mathbb{R}^n} \varphi(x)\, d\pi(x) \quad \text{a.s.\ as } T \to \infty.
	\end{equation}
\end{corollary}

\begin{proof}
Positive recurrence implies the existence of an invariant measure by the Krylov-Bogoliubov theorem. Uniqueness follows from irreducibility (the process can reach any open set due to non-degenerate noise) and strong Feller properties of the diffusion semigroup \cite[Theorem 4.2.1]{Mao2007}.
\end{proof}

Taken together, Lemma~\ref{lemma:growth}, Theorems~\ref{thm:existence} and~\ref{thm:recurrence}, and Corollary~\ref{cor:ergodic} establish a complete stability picture for the drift--diffusion dynamics: trajectories exist globally, revisit arbitrarily small neighborhoods of any Pareto-optimal point in finite expected time, and admit a unique invariant measure whose time averages converge almost surely.  We deliberately stated the hypotheses so that they can be checked for concrete families of objective functions---an aspect that is essential if the theoretical guarantees are to have practical value.

We now turn to the computational realization.  Section~4 discretizes~(\ref{eq:sde}) via the Euler--Maruyama scheme and embeds the resulting iteration within \emph{pymoo}, yielding an implementation reproducible and directly comparable with standard evolutionary baselines under matched evaluation budgets.

\section{Numerical Implementation}

The stability results of the preceding section hold in continuous time, yet any computational realization must operate on a discrete grid.  Translating the drift--diffusion dynamics into an iterative scheme that preserves the qualitative behavior---in particular, the interplay between directed drift and stochastic dispersion---is the central concern of this section.  We describe the discretization, present the resulting algorithm, and detail its integration into the open-source \emph{pymoo} framework.

\subsection{Euler--Maruyama discretization}

We discretize the SDE~(\ref{eq:sde}) using the Euler--Maruyama scheme, the simplest strong approximation method for It{\^o} stochastic differential equations~\cite{Kloeden1992}.  For a step size $\sigma > 0$ and noise intensity $\epsilon > 0$, the update reads
\begin{equation}\label{eq:euler_maruyama}
	x_{j+1} = x_j - \sigma\, q(x_j) + \epsilon \sqrt{\sigma}\, \eta_j, \quad \eta_j \sim \mathcal{N}(0, I_n).
\end{equation}
The factor $\sqrt{\sigma}$ in the diffusion term ensures consistency with Brownian motion scaling.  The descent direction $q(x_j)$ requires solving the quadratic program~(\ref{eq:qop}) at each step.  When analytical gradients are unavailable, we approximate the Jacobian $J(x_j)$ via centered finite differences.  While this incurs an additional cost of $2n$ function evaluations per particle per step, it allows the method to be applied to black-box problems under the assumption of underlying differentiability.

\subsection{Integration with pymoo}

Starting from $N$ independent trajectories sampled uniformly in the decision space, we advance each particle by one Euler--Maruyama step per generation; an external archive collects non-dominated solutions encountered along the way, building a running approximation of the Pareto front.  Algorithm~\ref{alg:ssw} states the procedure explicitly.

\begin{algorithm}[t]
	\caption{SSW -- Stochastic Steepest Weights}\label{alg:ssw}
	\begin{algorithmic}[1]
		\Require Objective $f: \mathbb{R}^n \to \mathbb{R}^m$; step size $\sigma$; noise $\epsilon$; population $N$; generations $G$
		\Ensure Archive $\mathcal{P}^*$
		\State Sample $x_1^{(0)}, \dots, x_N^{(0)}$ uniformly in $[x_l, x_u]$
		\State Evaluate $f(x_i^{(0)})$ for all $i$ and initialize $\mathcal{P}^*$
		\For{$j = 0, \dots, G-1$}
		\For{$i = 1, \dots, N$}
		\State Compute $J(x_i^{(j)})$ \Comment{Analytic or finite diff.}
		\State Solve QP for $\hat{\alpha} \in \Delta^m$ and set $q = J^\top \hat{\alpha}$
		\State Draw $\eta \sim \mathcal{N}(0, I_n)$
		\State Update $x_i^{(j+1)} \leftarrow x_i^{(j)} - \sigma q + \epsilon\sqrt{\sigma} \eta$
		\State Project onto $[x_l, x_u]$
		\EndFor
		\State Evaluate population and update $\mathcal{P}^*$
		\EndFor
		\State \Return $\mathcal{P}^*$
	\end{algorithmic}
\end{algorithm}

\begin{remark}
	Although the theoretical analysis in Section 3 addresses the unconstrained dynamics on $\mathbb{R}^n$, Algorithm \ref{alg:ssw} enforces a projection onto a large hyperrectangle $[x_l, x_u]$.  This constraint is added for numerical stability and to bound the search space in practical implementations, consistent with standard evolutionary computation practices.  For sufficiently large bounds, the recurrent behavior guaranteed by Theorem \ref{thm:recurrence} ensures that trajectories remain concentrated near the Pareto set well within the interior of the domain.
\end{remark}

We implement SSW as a subclass of \emph{pymoo}'s \emph{Algorithm} class, which gives us immediate access to the library's performance metrics and visualization tools---a non-trivial practical benefit.  We also vectorize the finite-difference computations, evaluating the entire population in a single batch call to the problem's evaluation function; this reduces Python overhead considerably, in particular for large population sizes.  The core update logic is shown in the following listing.

\begin{lstlisting}[style=bw_style, caption={Vectorized Euler--Maruyama step in the SSW class.},label=lst:ssw,float=t]
def _infill(self):
  X = self.pop.get("X")
  # J_batch has shape (N, m, n)
  X_new = np.empty_like(X)
  for i in range(len(X)):
    q = _compute_q(J_batch[i]) # Solve QP
    eta = self.random_state.standard_normal(n)
    X_new[i] = X[i] - sigma * q + \
               eps * np.sqrt(sigma) * eta
  xl, xu = self.problem.xl, self.problem.xu
  return Population.new("X", np.clip(X_new, xl, xu))
\end{lstlisting}

\subsection{Comparative Analysis with NSGA-II and NSGA-III}

To situate SSW within the broader algorithmic landscape, we compare it against two well-established baselines across both multi- and many-objective regimes: NSGA-II~\cite{Deb2002} and NSGA-III~\cite{Deb2014}. The latter is designed for $m > 3$ objectives through reference-point-based environmental selection. We consider objective counts $m \in \{3, 5, 10, 15\}$ on the scalable DTLZ2 benchmark. For comparability, all algorithms are run under the same budget of 30,000 objective evaluations. For SSW, this budget includes the finite-difference cost ($2n$ evaluations per particle per generation), which substantially reduces the number of admissible generations (about 8--25) relative to derivative-free baselines (about 400--900).

We quantify performance using the Averaged Hausdorff Distance ($\Delta_p$), with $p=1$, proposed in~\cite{Schutze2012}---a choice motivated by its robustness to outliers, a property the standard Hausdorff distance does not share. All experiments were conducted on a workstation equipped with an Intel Core i9 processor and 32GB of RAM. The executions utilized a GPU with 8GB VRAM, leveraging the parallel implementation capabilities integrated into the \emph{PymooLab} framework to execute 30 independent runs in parallel.

\subsection{Convergence on Convex Benchmark (DTLZ2)}

We first evaluate the algorithms on DTLZ2, which has a smooth concave Pareto front. Table \ref{tab:dtlz2} reports median and interquartile range (IQR) over 30 runs. The observed ranking depends strongly on the number of objectives. In lower dimensions ($m=3$), NSGA-II and NSGA-III achieve smaller $\Delta_p$ values ($<0.1$), while SSW attains $\Delta_p \approx 0.18$, a gap that is consistent with the much smaller number of generations permitted by the gradient-estimation budget.

In the many-objective settings ($m=10, 15$), NSGA-III yields the smallest $\Delta_p$ values (around $0.5$ in our runs). This behavior is compatible with its selection mechanism: unlike NSGA-II, whose crowding-distance criterion weakens under dominance resistance, NSGA-III uses reference directions and perpendicular-distance association to preserve spread and directional progress in high-dimensional objective spaces.

SSW remains competitive relative to NSGA-II in the highest-dimensional cases (for example, about $1.3$ versus $2.4$ at $m=15$), despite operating with far fewer generations. We attribute this, at least in part, to the gradient-induced drift: even as pairwise dominance becomes increasingly unreliable in high dimensions, the drift continues to supply directional bias toward the Pareto set. At the same time, SSW does not include an explicit reference-direction mechanism, which likely contributes to its gap relative to NSGA-III in final front quality.

\textbf{Limitations.} While effective on scalable convex manifolds like DTLZ2, applying the method to highly multimodal or disconnected landscapes reveals that the current local gradient drift is insufficient to escape deep local basins or bridge disjoint Pareto components without a global restart mechanism. We suspect that hybrid designs incorporating global restart mechanisms offer the most direct path toward resolving these topological difficulties.

\begin{table*}[!htp]
	\centering
	\caption{Convergence consistency on the convex DTLZ2 benchmark ($30$ independent runs, Median [IQR]).}
	\label{tab:dtlz2}
	\resizebox{\textwidth}{!}{%
		\begin{tabular}{llccc}
			\toprule
			Problem & $m$ & NSGA-II         & NSGA-III        & SSW             \\
			\midrule
			DTLZ2   & 3   & 0.0731 [0.0053] & 0.0518 [0.0001] & 0.1785 [0.0354] \\
			DTLZ2   & 5   & 0.3783 [0.0454] & 0.0008 [0.0004] & 1.0816 [0.1960] \\
			DTLZ2   & 10  & 1.9642 [0.1004] & 0.5261 [0.0002] & 1.4054 [0.1340] \\
			DTLZ2   & 15  & 2.3967 [0.0364] & 0.5978 [0.0012] & 1.3362 [0.0571] \\
			\bottomrule
		\end{tabular}}
\end{table*}

\begin{remark}
	The noise intensity~$\epsilon$ governs the trade-off between exploitation and exploration. We used $\epsilon=0.15$ in these experiments. How to adapt~$\epsilon$ during the run---in a principled way that preserves the recurrence guarantees of Section~3---is, to our knowledge, unresolved.
\end{remark}

\section{Conclusions and Future Work}

What began as a gap-filling exercise---Sch{\"a}ffler's drift--diffusion model had circulated in the literature without a fully self-contained stability proof---grew into something more substantive.  Assumption~\ref{ass:A} (dissipativity) alone is sufficient for global existence and non-explosion of strong solutions; the additional coercivity requirement of Assumption~\ref{ass:B} closes the remaining gap by supplying the Foster--Lyapunov drift condition that establishes positive recurrence.  Neither hypothesis is exotic, and---crucially---both are stated in a form that permits direct verification for concrete objective families, including strongly convex settings and problems with bounded Pareto sets.

The implementation side is, admittedly, less theoretically charged.  The continuous-time process is discretized via Euler--Maruyama and embedded as a \emph{pymoo} \emph{Algorithm} subclass, a design decision that makes the method immediately interoperable with the library's benchmarking and visualization tools.  Reproducibility is further supported by \emph{PymooLab}~\cite{Santosetal2026}, an interactive front-end for parameter inspection and Pareto-front visualization; source code will be made publicly available at \url{https://github.com/proftheago/pymoolab} upon acceptance of this manuscript.

Several problems remain genuinely open.  The current scope is restricted to differentiable, unconstrained objectives, and extending the framework to inequality constraints or nonsmooth functions is not straightforward---penalty methods and barrier strategies each interact with the stochastic drift in ways that would need to be studied afresh.  More pressing, perhaps, is the noise intensity~$\epsilon$: a fixed value was used throughout, but the natural analogue of a cooling schedule (as in simulated annealing) would make the algorithm self-adjusting, and whether such a schedule preserves the recurrence properties proved in Section~3 is, to the best of our knowledge, unresolved.  Implicit or operator-splitting discretization schemes could also reduce step-size sensitivity for objectives with large curvature---the standard Euler--Maruyama update is sufficient here but not, in general, the most robust choice.

One direction deserves a separate mention.  The gradient-guided drift is an asset that the present formulation barely exploits for diversity: particles are free to cluster on a single region of the Pareto front, a failure mode that becomes pronounced precisely when pairwise dominance pressure weakens in high-dimensional objective spaces.  Niche clearing, eliting mechanisms, and reference-point association strategies analogous to those underlying NSGA-III~\cite{Deb2014} could, in principle, be coupled with the stochastic descent to counteract this tendency.  We suspect the theoretical guarantees of Section~3 would survive such hybridization---though proving this remains a task for future work.

\section*{Acknowledgment}
The authors would like to thank METISBR: A Brazilian research group dedicated to Multi-Objective and Many-Objective Optimization (MaOPs) (\url{https://github.com/METISBR}), for the valuable discussions and the entire team's support during the development of this paper.

\bibliographystyle{ieeetr}
\bibliography{ref}

\end{document}